\documentclass{article}
\usepackage{amssymb}

\usepackage{makeidx}
\usepackage{graphicx}
\usepackage{amsmath}
\usepackage[english]{babel}

\newtheorem{theorem}{Theorem}

\newtheorem{corollary}[theorem]{Corollary}

\newtheorem{lemma}[theorem]{Lemma}

\newenvironment{proof}[1][Proof]{\textbf{#1.} }{\ \rule{0.5em}{0.5em}}

\begin{document}

\title{A Dutch Book theorem for partial subjective probability}
\author{Maurizio Negri, \\
Universit\`{a} di Torino}
\maketitle
\begin{abstract}

The aim of this paper is to show that partial probability can be justified
from the standpoint of subjective probability in much the same way as
classical probability does. The seminal works of Ramsey and De Finetti have
furnished a method for assessing subjective probabilities: ask about the
bets the decision-maker would be willing to place. So we introduce
the concept of partial bet and partial Dutch Book and prove for partial probability
a result similar to the Ramsey-De Finetti theorem. Finally, we make a comparison 
between two concepts of bet: we can bet our money on a sentence describing an event, or we can
bet our money on the event itself, generally conceived as a set. These two
ways of understanding a bet are equivalent in classical probability, but not in
partial probability.

\vspace{5mm} \noindent \textbf{Keywords:} Dutch Book theorem; non-classical probability; 
Kleene's Logic.
\end{abstract}

\section{Partial probability\label{parparprob}}

Classical probability theory is grounded on the concept of probability
space, a triple \ $(A,\mathcal{C}_{A},p)$ where $A$ is a sample space, $%
\mathcal{C}_{A}$ is a field of sets on $A$ and $p:\mathcal{C}_{A}\rightarrow
\lbrack 0,1]$ is a probability measure, i.e. a function satisfying
Kolmogoroff's axioms: 1) $p(A)=1$, 2) $p(X\cup Y)=p(X)+p(Y)$, when $X\cap
Y=\emptyset $. In this way probability is seen as the measure of an event
represented \ by a set of outcomes. Partial probability theory arises when
we substitute classical events with partial events (see \cite{negri2010}).
If $S$ is the sample space of an experiment, we define the set of all 
\textit{partial sets on} $S$ as the set $D(S)=\{(A,B):A,B\subseteq S,$ and $%
A\cap B=\emptyset \}$. Every $x\in A$ (resp.$B$) is said to be a \textit{%
positive} (resp. \textit{negative})\textit{\ }element of the partial set $%
(A,B)$. Partial sets are structured by the following operations: 
\begin{eqnarray*}
(A,B)\sqcap (C,D) &=&(A\cap C,B\cup D), \\
(A,B)\sqcup (C,D) &=&(A\cup C,B\cap D), \\
-(A,B) &=&(B,A), \\
0 &=&(\emptyset ,S), \\
1 &=&(S,\emptyset ), \\
n &=&(\emptyset ,\emptyset ).
\end{eqnarray*}
\noindent The algebra $\mathcal{D}(S)=(D(S),\sqcap ,\sqcup ,-,0,n,1)$ is the 
\textit{algebra of all partial sets} on $S$. A\textit{\ field of partial set}
is any subalgebra $\mathcal{G}_{S}\subseteq \mathcal{D}(S)$. A partial set $%
(A,B)$ is a \textit{Boolean partial set} if $A=S-B$. The set of Boolean
partial sets\ is a Boolean algebra isomorphic to $\mathcal{P}(S)$, the
classical power-set algebra. We define a binary relation between partial
sets setting $(A,B)\sqsubseteq (C,D)$ iff $A\subseteq C$ and $D\subseteq B$.
\noindent It can be easily proved that $(D(S),\sqsubseteq )$ is a partially
ordered set with $(S,\emptyset )$ as top and $(\emptyset ,S)$ as bottom
element.

When $S$ is a sample space, we say that $\mathcal{D}(S)$ is the \textit{%
algebra of partial events} on $S$. In relation with the experimental result $%
s\in S$, we say that $(A,B)$ occurs positively if $s\in A$, occurs
negatively if $s\in B$, is uncertain otherwise. Events of classical
probability theory are to be identified with Boolean partial sets and will
be called \textit{Boolean} or \textit{classical events}.

We define on $R^{2}$ the relation $\preceq $ setting $(x,y)\preceq (w,z)$
iff $x\leq w$ e $z\leq y$. (Note that the natural order of $R$ is reversed
on the second elements of the ordered pairs.) It can be easily shown that $%
(R^{2},\preceq )$ is a partially ordered set. The operations $x+y$, $xy$ and 
$-x$ in $R$ are extended pointwise to the product $R^{2}$, i.e. $%
(x,y)+(z,w)=(x+z,y+w)$, $(x,y)(z,w)=(xz,yw)$ and $-(x,y)=(-x,-y)$. \ We
define the set of \textit{partial probability values} $T$ as a subset of $%
R^{2}$, setting $T=\{(x,y)\in \lbrack 0,1]^{2}:$ $x+y\leq 1\}$. The
probability value of a partial event $(A,B)$ is a pair $(x,y)\in T$. The set
of partial probability values $T$ is partially ordered by $\preceq $ with a
maximum $(1,0)$ and a minimum $(0,1)$.

Given a partial field of set $\mathcal{G}_{S}$ on $S$ and a function $\mu
:G_{S}\rightarrow T$, we say that $\mu $ is a \textit{measure of partial
probability} \ when the following axioms are satisfied:

\begin{enumerate}
\item  $\mu (S,\emptyset )=(1,0)$.

\item  $\mu (A,B)+\mu (C,D)=\mu ((A,B)\sqcup (C,D))-\mu ((A,B)\sqcap (C,D)).$

\item  $\mu (-(A,B))=\sigma (\mu (A,B))$,

\item  $(\emptyset ,\emptyset )\sqsubseteq (A,B)$ implies $(0,0)\preceq \mu
(A,B)$,
\end{enumerate}

\noindent where $\sigma :[0,1]^{2}\rightarrow \lbrack 0,1]^{2}$ is defined
by $\sigma (x,y)=(y,x)$. \ As a consequence of axiom 4, we have $%
(0,0)\preceq \mu (A,\emptyset )$, for all $(A,\emptyset )\in \nabla $. A 
\textit{partial probability space} is a triple $(S,\mathcal{G}_{S},\mu )$
where $\mathcal{G}_{S}$ is a field of partial sets on the sample space $S$
and $\mu $ is a measure of partial probability.

Given a classical probability space $(S,\mathit{P}(S),p)$, we define the 
\textit{partial probability space associated to} $(S,\mathit{P}(S),p)$ as
follows: we define $\mu :D(S)\rightarrow T$ setting 
\begin{equation*}
\mu (A,B)=(p(A),p(B)).
\end{equation*}
As $(A,B)$ is a partial event, $A\cap B=\emptyset $ so $p(A)+p(B)=p(A\cup B)$%
: this proves that $\mu (A,B)\in T$. It can be easily proved that $\mu $
satisfies the four \ axioms above, so $(S,\mathcal{D}(S),\mu )$ is a partial
probability space. For instance, we can start from the classical probability
space $(S,\mathit{P}(S),p)$, where $S=\{n:1\leq n\leq 6\}$ and $p(n)=1/6$
for all $n\in S$. If $(A,B)$ is the Boolean event with $A=\{2,4,6\}$ and $%
B=\{1,3,5\}$, then $\mu (A,B)=(1/2,1/2)$, and if $(C,D)$ is the partial
event with $C=\{2,4\}$ and $D=\{5\}$, then $\mu (C,D)=(1/3,1/6)$. If the
outcome of the experiment is $3$, then we say that $(A,B)$ does not happen
and that $(C,D)$ is uncertain (neither happens nor does not happen). Some
events like $(A,\emptyset )$ have only positive occurrences (may only
happen), others like $(\emptyset ,A)$ have only negative occurrences, $%
(\emptyset ,\emptyset )$ is the absolutely undefined event.

We list some differences between classical and partial probability.

\begin{enumerate}
\item  $(A,B)$ or not-$(A,B)$ is no more the certain event, unless $(A,B)$
is Boolean, i.e. $B=S-A$. In general we can only say that $\mu ((A,B)\sqcup
-(A,B))=\mu (A\cup B,\emptyset )\succeq (0,0)$. For the same reason, $(A,B)$
and not-$(A,B)$ is not, in general, the impossible event.

\item  Additivity holds in the form of axiom 2, but now we have different
kinds of disjointness. In a field of partial sets $\mathcal{D}(S)$, the
partial sets like $(\emptyset ,X)$, belonging to the interval $[(\emptyset
,S),(\emptyset ,\emptyset )]$, are generalizations of the empty set (they
have no positive elements); so we can say that $(A,B)$ is disjoint from $%
(C,D)$ if $(A,B)\sqcap (C,D)=(\emptyset ,X)$, for some $X$. Now the partial
probability of a sum of disjoint sets is the sum of their probabilities, $%
\mu (A,B)+\mu (C,D)=\mu ((A,B)\sqcup (C,D))$, only when $(A,B)$ and $(C,D)$
have the maximum degree of disjointness $(\emptyset ,\emptyset )$.

\item  Partial probability values are only partially ordered by $\preceq $,
so we cannot say, in general, whether a partial event is more probable then
another partial event or not. However, it can be easily seen that the set of
partial probability values of Boolean events is totally ordered by $\preceq $%
.

\item  Conditioning on partial events is possible, but with a slightly
different meaning with respect to the classical case. See \cite{negri2010}
par. 4 and \cite{negri2013} par. 9, 13
\end{enumerate}

The aim of this paper is to show that partial probability can be justified
from the standpoint of subjective probability in much the same way as
classical probability does. The seminal works of Ramsey and De Finetti have
furnished a method for assessing subjective probabilities: ask about the
bets the decision-maker would be willing to place. In par. 2 we introduce
the concept of classical bet and shortly review the proof of Ramsey-De
Finetti theorem for classical probability. The remaining part of this paper
is devoted to prove an analogous result in the case of partial subjective
probability. To this scope we introduce in par. 3 the fundamentals concepts
of Kleene logic. In par. 4 we introduce the concept of partial bet and
partial Dutch Book and prove for partial probability a result similar to 
Ramsey-De Finetti theorem. In par. 5 we make a comparison between two concepts
of bet: we can bet our money on a sentence describing an event, or we can
bet our money on the event itself, generally conceived as a set. These two
ways of understanding a bet are equivalent in the classical case, but not in
the partial case.

\section{Classical bets\label{parclassbet}}

In the subjective approach to probability, the bearers of probability are no
more sets but sentences and the probability value is conceived as the degree
of belief in some event represented by a sentence. So we introduce a
sentential $n$-ary language $L_{n}$ based on the sentential variables $%
P_{n}=\{p_{1},...,p_{n}\}$, the connectives $\{\lnot ,\wedge ,\vee \}$ and
the constants $\{0,1\}$. (Through this text we deal only with finite
languages.) We denote with $F_{n}$ the set of formulas of $L_{n}$ and with $%
\mathcal{F}_{n}$ the algebra of formulas. (We simply write $L$, $P$ and $F$
when no confusion is possible.) We say that a function $\pi :F\rightarrow $ $%
[0,1]$ is a \textit{probability function}, if the following axioms are
satisfied:

\begin{enumerate}
\item  if $\models \alpha $ then $\pi (\alpha )=1$,

\item  if $\models \lnot (\alpha \wedge \beta )$ then $\pi (\alpha \vee
\beta )=\pi (\alpha )+\pi (\beta )$,
\end{enumerate}

\noindent where $\models $ is the consequence relation of bivalent logic.
(See, for instance, \cite{howson1989} or \cite{paris2005}.) Both the
measure-theoretic and the subjectivistic approach to classical probability
are deeply grounded on Boolean algebras. On one side the events of a
probability space constitute a Boolean algebra, on the other side the
definition of a probability function requires the semantics of bivalent
logic and Boolean truth-functions. Owing to this common ground, these two
ways of presenting probability can be translated one into the other (see,
for instance, \cite{negri2013}, par.4).

When does a function $b:F\rightarrow \lbrack 0,1]$, representing the
decision-maker's belief, satisfy the axioms of probability function? The
theorem that we are going to prove, commonly attributed to Ramsey and De
Finetti (see \cite{gillies2000}, p. 53), gives a sufficient condition
through the concept of bet. We understand a \textit{bet} as a triple $%
(\alpha ,x,r)$ where $\alpha \in F$, $x\in \lbrack 0,1]$ and $r\in R$. The
formula $\alpha $ represents the event on which the bet is placed, $x$ is
the betting quotient and $r$ the stake. By accepting this bet, we agree to
pay out $rx$ to receive $r$ from the bookmaker if the event described by $%
\alpha $ takes place, and nothing otherwise. The buyer's payoff can be
described by the following table: 
\begin{equation*}
\begin{array}{cc}
r-rx=r(1-x) & \text{if }\alpha \text{ is true in the actual world,} \\ 
-rx & \text{if }\alpha \text{ is false in the actual world.}
\end{array}
\end{equation*}
We can obtain the bookmaker's payoff by changing the sign of the stake. In
general, when we talk of `payoff', we understand the buyer's payoff.

The payoff is a function of $\alpha $, $x$, $r$ and of the state of the
world that makes $\alpha $ true or false. The definition of such a function
requires some basic concepts in the semantics of classical logic. We
identify with $2^{n}$ the set of all possible worlds, where $2=\{0,1\}$ is
the set of truth-values of classical logic. Then, for every sentential
variable $p_{i}$, we can set $V_{w}(p_{i})=w_{i}$: the truth-value of $p_{i}$
in the world $w$. As $\mathcal{F}$ is the absolutely free algebra, we know
that $V_{w}$ can be extended to a homomorphism $V_{w}:\mathcal{F}\rightarrow 
\mathbf{2}$, where $\mathbf{2}$ denotes the two-element Boolean algebra: $%
V_{w}(\alpha )$ is the truth-value of $\alpha $ in the world $w$. We say
that $\alpha $ is a consequence of $\gamma $, in symbols $\gamma \models
\alpha $, when $V_{w}(\gamma )\leq V_{w}(\alpha )$ for all $w\in 2^{n}$. We
say that $\alpha $ is a consequence of the set of formulas $\Gamma $, and
write $\Gamma \models \alpha $, when $\inf \{V_{w}(\gamma ):\gamma \in
\Gamma \}\leq V_{w}(\alpha )$ for all $w\in 2^{n}$. When $\Gamma =\emptyset $
we write simply $\models \alpha $ and this amounts to $V_{w}(\alpha )=1$,
for all $w\in 2^{n}$, because $\inf \{V_{w}(\gamma ):\gamma \in \emptyset
\}=\inf (\emptyset )=1$. We say that two formulas $\alpha $ and $\beta $ are 
\textit{equivalent} and write\textit{\ }$\alpha \equiv \beta $ iff $\alpha
\models \beta $ and $\beta \models \alpha $.

The semantics of classical logic can also be defined through a function that
associates to every formula $\alpha $ its meaning $M(\alpha )$, conceived as
the set of worlds in which $\alpha $ holds true. To this end we define $%
M(p_{i})=\{w\in 2^{n}:w_{i}=1\}$, for all sentential variable $p_{i}$:
intuitively, $M(p_{i})$ is the set of all worlds making $p_{i}$ true. As $%
\mathcal{F}$ is the absolutely free algebra, we can extend $M$ to a
homomorphism from $\mathcal{F}$ to the Boolean algebra $\mathcal{P}(2^{n})$.
\ The functions $V_{w}$ and $M$ are related in this way: given $V_{w}$, we
can define $M(\alpha )$ as $\{w\in 2^{n}:V_{w}(\alpha )=1\}$ and given $M$
we can define $V_{w}(\alpha )=1$ iff $\ w\in M(\alpha )$ and $V_{w}(\alpha
)=0$ iff $\ w\notin M(\alpha )$. So $\alpha \models \beta $ iff $M$($\alpha
)\subseteq M(\beta )$ and $\alpha \equiv \beta $ iff $M(\alpha )=M(\beta )$.

Now we return to the definition of the payoff. We associate to every bet $%
(\alpha ,x,r)$ a function $[\alpha ,x,r]$:$2^{n}\rightarrow R$ that gives
the \textit{payoff of} $(\alpha ,x,r)$ \textit{in the world} $w$ by 
\begin{equation*}
\lbrack \alpha ,x,r](w)=r(V_{w}(\alpha )-x)=\left\{ 
\begin{array}{ll}
r(1-x) & \text{if }V_{w}(\alpha )=1 \\ 
-rx & \text{if }V_{w}(\alpha )=0\text{.}
\end{array}
\right.
\end{equation*}
This definition can be naturally extended to finite sets of bets as follows.
If $B=\{(\alpha _{i},x_{i},r_{i}):1\leq i\leq n\}$, then we define the
payoff of the set $B$ in the world $w$ as 
\begin{equation*}
\lbrack B](w)=\overset{n}{\underset{i=1}{\sum }}\{[\alpha
_{i},x_{i},r_{i}](w)\}.
\end{equation*}
We say that a set of bets $B$ is a \textit{Dutch Book} if, for all $w\in
2^{n}$,$\ [B](w)<0$. Intuitively, a Dutch Book is a set of bets that
guarantees a sure loss for the buyer. There is, however, a weaker notion of
Dutch Book that will be useful in the following, where the set of bets can
at best break even and, in at least one possible world, has a net loss. We
say that a set of bets $B$ is a \textit{Weak Dutch Book} if, for all $w\in
2^{n}$,$\ [B](w)\leq 0$ and in at least one world $w$, $[B](w)<0$. (This
kind of Dutch Book has been introduced in \cite{shimony1955}.) If we take
the values of $b:F\rightarrow \lbrack 0,1]$ as betting quotients, i.e. the
bets in $B$ are of kind $(\alpha _{i},b(\alpha _{i}),r)$, and if a Dutch
Book is available with such betting rates, then we say that there is a 
\textit{Dutch Book for} $b.$

\begin{lemma}
For all $w\in 2^{n}$, $V_{w}(\alpha \vee \beta )=V_{w}(\alpha )+V_{w}(\beta
)-V_{w}(\alpha \wedge \beta )$.\label{lem1}
\end{lemma}

\begin{proof}
The proof is an easy calculation with truth-tables and is left to the reader.
\end{proof}

\begin{theorem}
\label{teodb}If $b:F\rightarrow \lbrack 0,1]$ and there is no Dutch Book for 
$b$, then $b$ is a probability function.
\end{theorem}

\begin{proof}
Axiom 1, if $\models \alpha $ then $b(\alpha )=1$. If $b(\alpha )<1$, then $%
B=\{(a,b(\alpha ),-1)\}$ is a Dutch Book because, for all $w\in 2^{n}$, $%
V_{w}(\alpha )=1$ and then 
\begin{equation*}
\lbrack B](w)=[\alpha ,b(\alpha ),-1](w)=-1(V_{w}(\alpha )-b(\alpha
))=b(\alpha )-1<0.
\end{equation*}
Axiom 2, if $\models \lnot (\alpha \wedge \beta )$ then $b(\alpha \vee \beta
)=b(\alpha )+b(\beta )$. Firstly we prove that 
\begin{equation}
\text{if}\models \lnot \alpha \text{ then }b(\alpha )=0.  \tag{1}  \label{ze}
\end{equation}
If $0<b(\alpha )$, \ then $B=\{(a,b(\alpha ),1)\}$ is a Dutch Book because,
for all $w\in 2^{n}$, $V_{w}(\alpha )=0$ and then $[B](w)=[\alpha
,x,1](w)=1(V_{w}(\alpha )-x)=0-x<0.$

Secondly we show that 
\begin{equation}
b(\alpha \vee \beta )+\beta (\alpha \wedge \beta )=b(\alpha )+b(\beta ). 
\tag{2}  \label{zeze}
\end{equation}
Let $b(\alpha )=x$, $b(\beta )=y$, $b(\alpha $ $\vee \beta )=z$, and $\beta
(\alpha \wedge \beta )=w$. If $z+w>x+y$, we define $B=\{(\alpha \vee \beta
,z,1),(\alpha \wedge \beta ,w,1),(\alpha ,x,-1),(\beta ,y,-1)\}$. We show
that $B$ is a Dutch Book. For all $w\in 2^{n}$, we have 
\begin{eqnarray*}
\lbrack B](w) &=&(V_{w}(\alpha \vee \beta )-z)+(V_{w}(\alpha \wedge \beta
)-w)+(x-V_{w}(\alpha ))+(y-V_{w}(\beta )) \\
&=&(V_{w}(\alpha \vee \beta )+V_{w}(\alpha \wedge \beta )-V_{w}(\alpha
)-V_{w}(\beta ))+(x+y-z-w) \\
&=&x+y-z-w \\
&<&0,
\end{eqnarray*}
because $V_{w}(\alpha \vee \beta )+V_{w}(\alpha \wedge \beta )-V_{w}(\alpha
)-V_{w}(\beta )=0$ by the above lemma. If $z+w<x+y$, just change the sign of
the stakes.

Now we can easily prove axiom 2: if $\models \lnot (\alpha \wedge \beta )$
then $b(\alpha \wedge \beta )=0$, by \ref{ze}, so $b(\alpha \vee \beta
)=b(\alpha )+b(\beta )$ follows from \ref{zeze}.
\end{proof}

\section{Partial subjective probability and Kleene logic\label{parparsub}}

As we have seen in par. \ref{parclassbet}, classic probability can be
understood as a degree of belief in a sentence. We can do the same with
partial probability: we have only to shift from bivalent logic to Kleene
logic and from probability values in $[0,1]$ to probability values in $T$.
This should justify a brief digression in the semantics of Kleene logic.

The language of Kleene \ $n$-ary logic is $L_{n}^{\ast }=$ $L_{n}\cup \{n\}$%
, where $L_{n}$ is the $n$-ary language of classical logic introduced in
par. \ref{parclassbet}. We denote with $F_{n}^{\ast }$ the set of formulas
of $L_{n}^{\ast }$ and with $\mathcal{F}_{n}^{\ast }$ the algebra of $n$-ary
formula. We write simply $L^{\ast }$, $F^{\ast }$ and $\mathcal{F}^{\ast }$
when no confusion is possible. We denote with $K$ the set $\{0,n,1\}$ of
truth-values of Kleene logic, where $n$ stands for `neutral' or `uncertain'.
Every $w\in K^{n}$ cas be seen as an instantaneous description of the world,
at the level of the atomic facts represented by sentential variables, so we
can define a function $V_{w}$ from $\{p_{i}:i\leq n\}$ to $K$ setting $%
V_{w}(p_{i})=w_{i}$. When $w_{i}=n$, the atomic fact represented by $p_{i}$
neither happens nor does not happen. This uncertainty may be of an epistemic
kind, related to a lack of knowledge, or may be deeply rooted in the
reality. The next step is extending $V_{w}$ to all formulas an to this end
we give an algebraic structure to the set of truth-values as follows: we
define on $K$ the total order $0<n<1$ and define $x\wedge y=\min (x,y)$ and $%
x\vee y=\max (x,y)$, what amounts to giving the following truth-tables:

\begin{center}
\begin{tabular}{c|ccc}
$\wedge $ & 1 & 0 & n \\ \hline
1 & 1 & 0 & n \\ 
0 & 0 & 0 & 0 \\ 
n & n & 0 & n
\end{tabular}
\ \ \ \ 
\begin{tabular}{c|ccc}
$\vee $ & 1 & 0 & n \\ \hline
1 & 1 & 1 & 1 \\ 
0 & 1 & 0 & n \\ 
n & 1 & n & n
\end{tabular}
\end{center}

\noindent As for negation, we set $\lnot (n)=n$, $\lnot (0)=1$, $\lnot (1)=0$%
. Finally, we denote with $\mathcal{K}$ the algebra $(K,\wedge ,\vee ,\lnot
,0,1,n)$. As $\mathcal{F}^{\ast }$ is the absolutely free algebra, we can
extend $V_{w}$ to an homomorphism $V_{w}:\mathcal{F}\rightarrow \mathcal{K}$%
. So we say that $\alpha $ is \textit{true in} $w$ iff $V_{w}(\alpha )=1$, 
\textit{false} if $V_{w}(\alpha )=0$ and \textit{neutral} if $V_{w}(\alpha
)=n$. We say that two formulas $\alpha $ and $\beta $ are \textit{equivalent}
and write\textit{\ }$\alpha \equiv \beta $ iff $V_{w}(\alpha
)_{w}=V_{w}(\beta )_{w}$, for all $w\in K^{n}$.

The semantics of Kleene logic can also be defined through a function that
associates to every formula $\alpha $ its meaning $M(\alpha )$ conceived as
a partial set in $D(K^{n})$, where $M(\alpha )_{0}$ and $M(\alpha )_{1}$ are
the set of worlds in which $\alpha $ is respectively true and false, the
positive and the negative models of $\alpha $. Firstly, we define a function 
$M:\{p_{i}:i\leq n\}\rightarrow D(K^{n})$ setting

\begin{equation*}
M(p_{i})=(\{w\in K^{n}:w_{i}=1\},\{w\in K^{n}:w_{i}=0\}).
\end{equation*}
As $\mathcal{F}^{\ast }$ is free, we extend $M$ to a homomorphism $M:%
\mathcal{F}\rightarrow \mathcal{D}(K^{n})$. Then the meanings of formulas
can be recursively defined by the following equations:

\begin{eqnarray*}
M(\alpha \wedge \beta ) &=&M(\alpha )\sqcap M(\beta ), \\
M(\alpha \vee \beta ) &=&M(\alpha )\sqcup M(\beta ), \\
M(\lnot \alpha ) &=&-M(\alpha ), \\
M(0) &=&(\emptyset ,K^{n}), \\
M(1) &=&(K^{n},\emptyset ), \\
M(n) &=&(\emptyset ,\emptyset ).
\end{eqnarray*}

These two ways of giving a semantics are equivalent: if we take the notion
of meaning given by $M$ as primitive, then we can define 
\begin{equation*}
V_{w}(\alpha )=\left\{ 
\begin{array}{ccc}
0 & \mathrm{if} & w\in M(\alpha )_{1}, \\ 
1 & \mathrm{if} & w\in M(\alpha )_{0}, \\ 
n & \text{\textrm{if}} & w\in K^{n}-(M(\alpha )_{0}\cup M(\alpha )_{1});
\end{array}
\right.
\end{equation*}
if we take the notion of truth in the possible world $w$ given by $V_{w}$ as
primitive, then 
\begin{equation*}
M(\alpha )=(\{w\in K^{n}:V_{w}(\alpha )=1\},\{w\in K^{n}:V_{w}(\alpha )=0\}).
\end{equation*}
We define the notion of \textit{logical consequence} as follows: $\alpha
\models \beta $ iff $M(\alpha )\sqsubseteq M(\beta )$ iff $M(\alpha
)_{0}\subseteq M(\beta )_{0}$ and $M(\beta )_{1}\subseteq M(\alpha )_{1}$
iff every positive model of $\alpha $ is a positive model of $\beta $ and
every negative model of $\beta $ is a negative model of $\alpha $. The
notion of logical consequence can be generalized to 
\begin{equation*}
\Gamma \models \alpha \text{ iff }{\sqcap }\{M(\gamma ):\gamma \in \Gamma
\}\sqsubseteq M(\alpha ).
\end{equation*}
In terms of $V_{s}$, the definition runs as follows: 
\begin{equation*}
\Gamma \models \alpha \text{ iff, for all }s\in K^{n}\text{, }\inf
\{V_{s}(\gamma ):\gamma \in \Gamma \}\leq V_{s}(\alpha ).
\end{equation*}
As in the case of classical logic, an easy calculation based on truth-tables
gives the following theorem.

\begin{theorem}
For all $w\in K^{n}$, $V_{w}(\alpha \vee \beta )=V_{w}(\alpha )+V_{w}(\beta
)-V_{w}(\alpha \wedge \beta )$.\label{lem2}
\end{theorem}

Now we can define the notion of partial probability function. As the axioms
of probability function (see par. \ref{parclassbet}) are modeled on
Kolmogoroff's axioms, so the axioms of partial probability function are
modeled on the axioms of measure of partial probability introduced in par. 
\ref{parparprob}. We say that $\pi :F^{\ast }\rightarrow T$ is a \textit{%
partial probability function} if the following axioms are satisfied, where $%
\models $ denotes logical consequence in Kleene logic:

\begin{enumerate}
\item  $1\models \alpha $ implies $\pi (\alpha )=(1,0)$,

\item  $\pi (\alpha \vee \beta )=\pi (\alpha )+\pi (\beta )-\pi (\alpha
\wedge \beta )$,

\item  $\pi (\lnot \alpha )=\sigma (\pi (\alpha ))$,

\item  $n\models \alpha $ implies $(0,0)\preccurlyeq \pi (\alpha ).$
\end{enumerate}

We call $\pi (\alpha )$ the \textit{partial degree of belief} in $\alpha $
of our decision-maker. The following theorem shows some fundamental
properties of $\pi $.

\begin{theorem}
\label{teoprobfun}If $\pi $ is a partial probability function on $%
L_{n}^{\ast }$, then

\begin{enumerate}
\item  $\pi (n)=(0,0).$

\item  $\alpha \models 0$ implies $\pi (\alpha )=(0,1)$,

\item  $\alpha \models n$ implies $\pi (\alpha )\preccurlyeq (0,0)$,

\item  $\pi (\alpha )=\pi (\alpha \vee n)+\pi (\alpha \wedge n)$.
\end{enumerate}
\end{theorem}

\begin{proof}
1. From $n\models n$ we have $(0,0)\preccurlyeq \pi (n)$, by axiom 4. In
Kleene logic we have $n\models \lnot n$ and so $(0,0)\preccurlyeq \pi (\lnot
n)$. In general we have $(x,y)\preccurlyeq (x^{\prime },y^{\prime })$ iff $%
\sigma (x^{\prime },y^{\prime })\preccurlyeq \sigma (x,y)$, thus $\sigma
(\pi (\lnot (n)))\preccurlyeq (0,0)$. By axiom 3, $\sigma (\pi (\lnot
(n)))=\sigma (\sigma (\pi (n)))=\pi (n)$ holds. Thus $\pi (n)\preccurlyeq
(0,0)$.

2. In Kleene logic $\alpha \models 0$ implies $1\models \lnot \alpha $ and
so, by axiom 1, $\pi (\lnot \alpha )=(1,0)$. Thus, by axiom 3, $\sigma (\pi
(\alpha ))=(1,0)$ e $\pi (\alpha )=(0,1)$.

3. In Kleene logic $\alpha \models n$ implies $n\models \lnot \alpha $, so $%
(0,0)\preccurlyeq \pi (\lnot \alpha )=\sigma (\pi (\alpha ))$ and $\pi
(\alpha )\preccurlyeq (0,0)$.

4. $\pi (\alpha \vee n)=\pi (\alpha )+\pi (n)-\pi (\alpha \wedge n)=\pi
(\alpha )-\pi (\alpha \wedge n)$ by axiom 2 and point 1).
\end{proof}

\section{Partial bets\label{parparzbets}}

A classical bet is a mechanism that receives a triple constituted by a
sentence $\alpha $ in the language $L$ of classical logic, a betting
quotient $x\in \lbrack 0,1]$ and a stake $r\in R$ as input, and gives a
payoff, represented by a real number as output.\ We observe that both
payoffs and stakes belong to $R$, besides, the set of truth-values of
sentences $\{0,1\}$, betting rates $[0,1]$ and \ stakes $R$ are related by $%
\{0,1\}\subseteq $ $[0,1]\subseteq R$. As a result we have been able to
define the payoff \ $[\alpha ,x,r](w)$ of a classical bet in the world $w$
as $r(V_{w}(\alpha )-x)$.

When we consider a partial bet, we begin with a sentence $\alpha $ in the
language $L^{\ast }$ of Kleene and a betting quotient $(x,y)\in T$. The
first step is finding a common ground for truth-values of Kleene logic and
partial probability values, so we make $K\subseteq T$ by giving a new
definition of $K$. So, from now on, we set $K=$ $\{(0,1),(0,0),(1,0)\}$.
After all, the essential nature of truth-values is immaterial, as long as
their formal properties remain unchanged. If we restrict to $K$ the partial
order $\preceq $ defined on $T$, we have $(0,1)\preceq (0,0)\preceq (1,0)$
where $(0,1)$ stands for `false', $(0,0)$ for `neutral' and $(1,0)$ for
`true'. We define the algebra $\mathcal{K}$ of truth values as above. For
reader's convenience we write the new truth-tables:

\begin{center}
\bigskip 
\begin{tabular}{c|ccc}
$\wedge $ & (1,0) & (0,1) & (0,0) \\ \hline
(1,0) & (1,0) & (0,1) & (0,0) \\ 
(0,1) & (0,1) & (0,1) & (0,1) \\ 
(0,0) & (0,0) & (0,1) & (0,0)
\end{tabular}
\ \ \ \ 
\begin{tabular}{c|ccc}
$\vee $ & (1,0) & (0,1) & (0,0) \\ \hline
(1,0) & (1,0) & (1,0) & (1,0) \\ 
(0,1) & (1,0) & (0,1) & (0,0) \\ 
(0,0) & (1,0) & (0,0) & (0,0)
\end{tabular}
\end{center}

\noindent As for $\lnot $, we have $\lnot (0,1)=(1,0)$ and $\lnot
(0,0)=(0,0) $. Finally we get a homomorphism $V_{w}:\mathcal{F}^{\ast }%
\mathcal{\rightarrow K}$ and $V_{w}(\alpha )$ will be the truth-value of $%
\alpha $ in the world $w\in K^{n}$. The proof of lemma \ref{lem2} remains
obviously unchanged.

The second step is finding a common ground for partial probability values
and stakes, so we define a stake as a pair $(h,k)$ in $R^{2}$. The payoff of
a partial bet will be an element of $R^{2}$ too. Stakes and payoffs are
ordered by the same partial order $\preccurlyeq $ defined on partial
probability values. So $R^{2}$ is the common ground for truth-values,
partial probability values, stakes and payoffs

In view of the particular partial ordering of payoffs, some considerations
are in order. (And the same holds for stakes.) A payoff $(h,k)$ has a
positive part $h$ and a negative part $k$. If $h>0$ then $h$ is a reward and
if $h<0$ then $h$ is a loss, \ but on the second elements the order is
reversed, so $k>0$ is a true loss and $k<0$ is a true reward. The positive
and negative parts may come from completely different domains, so we may
think the first value $h$ representing the gain or loss of money and the
second value $k$ representing a degree of physical pain or gratification.
What is essential is that we cannot in principle strike a balance between
the first and the second component of $(h,k)$ and reduce the pair to a
single number.

As $\preceq $ is a partial order, we cannot say in general whether $(h,k)$
is better than $(h^{\prime },k^{\prime })$, but we can make the following
distinctions. We can partition $R^{2}$ in three exhaustive and disjoint
subsets: the diagonal $\delta =\{(x,x):x\in R\}$, the pairs under the
diagonal $\delta ^{+}=$ $\{(x,y):x>y\}$ and the pairs over the diagonal $%
\delta ^{-}=$ $\{(x,y):x<y\}$. The elements of $\delta $ are `neutral
payoffs': if the first component $h$ is positive and so can be seen as a
gain, then the second element $h$ represents a loss of the same intensity.
The situation is reversed when $h$ is negative. The elements of $\delta ^{+}$
give more reward than \ punishment and the elements of $\delta ^{-}$ behave
in the opposite way, so a payoff in $\delta ^{+}$ can be seen as `good' and
\ a payoff in $\delta ^{-}$ as `bad'.

Now we can introduce the concept of \textit{partial bet} as a triple $%
(a,(x,y),(h,k))$ where $\alpha $ is a formula of $L_{n}^{\ast }$, $(x,y)\in
T $, $(h,k)\in R^{2}$. We can describe the payoff of the buyer \ for $%
(a,(x,y),(h,k))$ by the following table,

\begin{center}
\begin{tabular}{ll}
$(h,k)((1,0)-(x,y))=(h(1-x),-ky)$ & if $\alpha $ is true in the actual world
\\ 
$(h,k)((0,0)-(x,y))=(-hx,-ky)$ & if $\alpha $ is neutral in the actual world
\\ 
$(h,k)((0,1)-(x,y))=(-hx,k(1-y))$ & if $\alpha $ is false in the actual
world.
\end{tabular}
\end{center}

\bigskip The net gain varies as the actual world $w$ varies in the set $K^{n}
$ of the possible worlds, so we define a function $[\alpha ,(xy),(h,k)]$:$%
K^{n}\rightarrow R^{2}$ that gives the \textit{payoff of} $(\alpha
,(xy),(h,k))$ \textit{in the world} $w$ by 
\begin{equation*}
\lbrack \alpha ,(xy),(h,k)](w)=(h.k)(V_{w}(\alpha )-(x,y)).
\end{equation*}
We have $[\alpha ,(xy),(h,k)](w)=$ 
\begin{equation*}
\left\{ 
\begin{array}{ll}
(h,k)((1,0)-(x,y))=(h(1-x),-ky) & \text{if }V_{w}(\alpha )=(1,0) \\ 
(h,k)((0,0)-(x,y))=(-hx,-ky) & \text{if }V_{w}(\alpha )=(0,0) \\ 
(h,k)((0,1)-(x,y))=(-hx,k(1-y)) & \text{if }V_{w}(\alpha )=(0,1.)
\end{array}
\right. 
\end{equation*}
As in the case of classical bets, this definition can be naturally extended
to finite sets of bets as follows. If $B=\{(\alpha
_{i},(x_{i},y_{i}),(h_{i},k_{i})):1\leq i\leq n\}$, then we define the net
gain of the set $B$ in the world $w$ as 
\begin{equation*}
\lbrack B](w)=\overset{n}{\underset{i=1}{\sum }}\{[\alpha
_{i},(x_{i},y_{i}),(h_{i},k_{i})]\}.
\end{equation*}
The definition of Dutch Book is slightly different from the classical case.
The payoffs of classical bets are in $R$ with its natural order $\leq $, so $%
0$ is an equilibrium point between loss, negative real numbers, and rewards,
positive natural numbers. The payoffs of partial bets are in $R^{2}$
partially ordered by $\preccurlyeq $, where $\delta $ is a natural
separation between loss, the pairs in $\delta ^{-}$, and reward, the pairs
in $\delta ^{+}$. So we say that a set of partial bets $B$ is a \textit{%
Dutch Book} if, for all $w\in 2^{n}$,$\ [B](w)\in \delta ^{-}$. Intuitively,
a Dutch Book is a set of bets that guarantees more loss than reward for the
buyer. We say that a set of partial bets $B$ is a \textit{Weak Dutch Book}
if, for all $w\in 2^{n}$,$\ [B](w)\in \delta ^{-}\cup \delta $ and there is
at least one possible world $w$ such that $[B](w)\in \delta ^{-}$: so in all
cases we have no gain and in some cases a sure loss. If $B$ is a Dutch Book
then $B$ is obviously a Weak Dutch Book too.

Given a function $b$ from the set \ $F^{\ast }$ of formulas of Kleene logic
to the set $T$ of partial probability values, we say that $B$ is a \textit{%
(Weak) Dutch Book for} $\pi $ if $B$ is a (Weak) Dutch Book and the bets in $%
B$ are of kind $(\alpha _{i},b(\alpha _{i}),(h_{i},k_{i}))$, where the
betting rates are given by $b$. Now we can prove a result similar to theorem 
\ref{teodb}.

\begin{lemma}
For all $(x,y),(z,w)\in T$, if $(y,x)\neq (z,w)$ and $x+z=y+w$, then there
are $h,h^{\prime },k,k^{\prime }\in R$ such that

\begin{enumerate}
\item  $hx+h^{\prime }z=ky+k^{\prime }w$,

\item  $h<k^{\prime }$ and $h^{\prime }<k$.
\end{enumerate}
\end{lemma}

\begin{proof}
Firstly, we observe that $x$, $y$, $z$ and $w$ all belong to $[0,1]$ because 
$(x,y)$ and $(z,w)$ belong to $T$. Secondly, if we set $k^{\prime }=q+h$ and 
$k=t+h^{\prime }$, then we reduce ourselves to prove that there are $%
q,t,h,h^{\prime }$, with $q,t>0$ such that 
\begin{equation}
hx+h^{\prime }z=(t+h^{\prime })y+(q+h)w.  \tag{1}  \label{one}
\end{equation}
Now we distinguish two cases.

Case1, $0<y,w$. The \ above equation can be rewritten as $%
qw+ty=h(x-w)-h^{\prime }(y-z)$. From our hypothesis $x+z=y+w$ we have $%
x-w=y-z$, so we obtain $qw+ty=h(x-w)-h^{\prime }(x-w)$ that can also be
written as 
\begin{equation}
q=-\frac{y}{w}t+\frac{(h-h^{\prime })(x-w)}{w},  \tag{2}  \label{two}
\end{equation}
a linear equation in two unknowns $q,t$ and parameters $h,h^{\prime },w,x,y$%
, with $w,y\neq 0$.This equation can be plotted as a line with negative
slope, as $t$ varies on the horizontal axis and $q$ on the vertical axis. We
can choose $h$ and $h^{\prime }$ such that the set of pairs $(t,q)$
satisfying (\ref{one}), with $t,q>0$, is not empty. To this scope we make
the $t$-intercept positive and the $q$-intercept positive, i.e. 
\begin{equation*}
t=\frac{(h-h^{\prime })(x-w)}{y}>0\text{ and }q=\frac{(h-h^{\prime })(x-w)}{w%
}>0,
\end{equation*}
by choosing : i) $h>h^{\prime }$ if $x>w$, ii) $h^{\prime }>h$ if $w>x$.
(Remember that $w,y\neq 0$.) With such a choice of $h$ and $h^{\prime }$,
any pair $(t,q)$ such that (\ref{two}) holds and 
\begin{equation*}
0\leq t\leq \frac{(h-h^{\prime })(x-w)}{y}\text{ and }0\leq q\leq \frac{%
(h-h^{\prime })(x-w)}{w}
\end{equation*}
gives a solution. (The case $x=w$ is impossible, because from $x=w$ and the
hypothesis $x+z=y+w$ we have $y=z$, that is contrary to our hypothesis $%
(y,x)\neq (z,w)$.)

Case 2, \ $y=0$, or $w=0$.

Subcase a), $y=0$. Firstly we prove that $w\neq 0$ and $z\neq 0$. If $w=0$
then, from our hypothesis $x+z=y+w$ we have $x+z=0$ and so $x=z=0$. Then $%
x,y,z,w$ are all $0$ and an absurd follows our hypothesis $(y,x)\neq (z,w)$.
If $z=0$ then $(x,y)=(x,0)$ and $(z,w)=(0,w)$, so from our hypothesis $%
(y,x)\neq (z,w)$ we have $x\neq w$. But $x+z=y+w$ holds by hypothesis, so $%
x+0=0+w$ and $x=w$, that is absurd. Now equation (\ref{one}) reduces $%
hx+h^{\prime }z=(q+h)w$, because $y=0$ by hypothesis, that can be rewritten
as 
\begin{equation*}
h(x-w)+h^{\prime }z=qw.
\end{equation*}
By hypothesis, $x+z=y+w=w$, so $x-w=-z$, so the above equation becomes \ $%
-hz+h^{\prime }z=qw$, that can be rewritten as 
\begin{equation*}
q=\frac{(h^{\prime }-h)z}{w},
\end{equation*}
where $w>0$ and $z>0$, so we get $q>0$ by choosing $h^{\prime }>h$.

Subcase b), $w=0$. The proof is similar to subcase a).
\end{proof}

\begin{theorem}
If $b:F\rightarrow T$ and there is no Weak Dutch Book for, then $b$ is a
partial probability function.
\end{theorem}

\begin{proof}
We prove that $b$ satisfies the four axioms of partial probability function.

Axiom 1. We must prove that $1\models \alpha $ implies $b(\alpha )=(1,0)$,
so we suppose $b(\alpha )=(x,y)\neq (1,0)$ and set $B=\{(\alpha
,(x,y),(-1,-11))\}$. We have, for all $w\in K^{n}$, 
\begin{equation*}
\lbrack B](w)=(-1,-1)(V_{w}(\alpha )-(x,y))=(-1,-1)((1,0)-(x,y))=(x-1,y).
\end{equation*}
From $(x,y)\neq (1,0)$ we have $x<1$. (As $(x,y)\in T$, $(x,y)\neq (1,0)$
means $(x,y)\prec (1,0)$,  so $x<1$.) So we have $x-1<0$ and $0\leq y$ and
then $(x-1,y)\prec (0,0)$. So $[B](w)\in \delta ^{-}$ for all $w\in K^{n}$%
and $B$ is a Dutch Book for $b$.

Axiom 2. We must prove that $b(\alpha \vee \beta )=b(\alpha )+b(\beta
)-b(\alpha \wedge \beta )$.

Case 1:  $b(\alpha )+b(\beta )\prec b(\alpha \vee \beta )+b(\alpha \wedge
\beta )$. We set 
\begin{eqnarray*}
B &=&\{(\alpha \vee \beta ,b(\alpha \vee \beta ),(1,1)),((\alpha \wedge
\beta ,b(\alpha \wedge \beta ),(1,1)), \\
&&(\alpha ,b(\alpha ),(-1,-1)),(\beta ,b(\beta ),(-1,-1))\}.
\end{eqnarray*}

Then we have, for all $w\in K^{n}$, 
\begin{eqnarray*}
\lbrack B](w) &=&(1,1)(V_{w}(\alpha \vee \beta )-b(\alpha \vee \beta
))+(1,1)(V_{w}(\alpha \wedge \beta )-b(\alpha \wedge \beta ))+ \\
&&(-1,-1)(V_{w}(\alpha )-b(\alpha ))+(-1,-1)(V_{w}(\beta )-b(\beta )) \\
&=&V_{w}(\alpha \vee \beta )-b(\alpha \vee \beta )+V_{w}(\alpha \wedge \beta
)-b(\alpha \wedge \beta )+b(\alpha )- \\
&&V_{w}(\alpha )+b(\beta )-V_{w}(\beta ) \\
&=&(V_{w}(\alpha \vee \beta )+V_{w}(\alpha \wedge \beta )-V_{w}(\alpha
)-V_{w}(\beta ))+ \\
&&(b(\alpha )+b(\beta )-b(\alpha \vee \beta )-b(\alpha \wedge \beta )) \\
&=&b(\alpha )+b(\beta )-b(\alpha \vee \beta )-b(\alpha \wedge \beta ) \\
&\prec &(0,0)
\end{eqnarray*}
where the next to last line follows from lemma \ref{lem2} and the last line
from our hypothesis. As $[B](w)\in \delta ^{-}$ for all $w\in K^{n}$, $B$ is
a Dutch Book for $b$.

Case 2: $b(\alpha \vee \beta )+b(\alpha \wedge \beta )\prec b(\alpha
)+b(\beta )$. We set 
\begin{eqnarray*}
B &=&\{(\alpha \vee \beta ,b(\alpha \vee \beta ),(-1,-1)),((\alpha \wedge
\beta ,b(\alpha \wedge \beta ),(-1,-1)), \\
&&(\alpha ,b(\alpha ),(1,1)),(\beta ,b(\beta ),(1,1))\}
\end{eqnarray*}
and we get $[B](w)=b(\alpha \vee \beta )+b(\alpha \wedge \beta )-b(\alpha
)-b(\beta )\prec (0,0)$, so $[B](w)\in \delta ^{-}$.

Axiom 3. We must prove that $b(\lnot \alpha )=\sigma (b(\alpha ))$, so we
suppose that $b(\lnot \alpha )\neq \sigma (b(\alpha ))$. Let $b(a)=(x,y)$
and $b(\lnot \alpha )=(z,w)$, where $(z,w)\neq (y,x)=\sigma (b(\alpha ))$.
We distinguish three cases.

Case 1, $x+z<y+w$. If we set 
\[
B=\{(\alpha ,(x,y),(-1,-1)),(\lnot \alpha
,(z,w),(-1,-1))\},
\]
then we have, for all $w\in K^{n}$, 
\begin{eqnarray*}
\lbrack B](w) &=&(-1,-1)(V_{w}(\alpha )-(x,y))+(-1,-1)(V_{w}(\lnot \alpha
)-(z,w)) \\
&=&(-1,-1)(V_{w}(\alpha )+\sigma (V_{w}(\alpha ))-(x+z,y+w)).
\end{eqnarray*}
We note that 
\begin{equation*}
V_{w}(\alpha )+\sigma (V_{w}(\alpha ))=\left\{ 
\begin{array}{ll}
(1,1) & \text{if }V_{w}(\alpha )=(1,0)\text{ or }V_{w}(\alpha )=(0,1)\text{,}
\\ 
(0,0) & \text{if }V_{w}(\alpha )=(0,0)\text{,}
\end{array}
\right. 
\end{equation*}
so we have 
\begin{equation*}
\lbrack B](w)=\left\{ 
\begin{array}{l}
(-1,-1)((1,1)-(x+z,y+w))=((x+z)-1,(y+w)-1)\text{,} \\ 
(-1,-1)((0,0)-(x+z,y+w))=(x+z,y+w)\text{.}
\end{array}
\right. 
\end{equation*}
By hypothesis $x+z<y+w$, so in both cases $[B](w)\in \delta ^{-}$ and $B$ is
a Dutch Book.

Case 2, $y+w<x+z$. We take $B$ as in case 1, but change the stake $(-1,-1)$
in $(1,1)$, then we have 
\begin{equation*}
\lbrack B](w)=\left\{ 
\begin{array}{l}
(1-(x+z),1-(y+w)\text{,} \\ 
(-(x+z),-(y+w))\text{.}
\end{array}
\right. 
\end{equation*}
In both cases we have a point in $\delta ^{-}$ because, by hypothesis, $%
-(x+z)<-(y+w)$.

Case 3, $x+z=y+w$. We show that there are stakes $(h,k)$ and $(h^{\prime
},k^{\prime })$ such that, setting 
\begin{equation*}
B=\{(\alpha ,(x,y),(h,k)),(\lnot \alpha ,(z,w),(h^{\prime },k^{\prime }))\},
\end{equation*}
we have $[B](w)\in \delta ^{-}$ for all $w\in K^{n}$. Firstly we observe
that 
\begin{eqnarray*}
\lbrack B](w) &=&(h,k)(V_{w}(\alpha )-(x,y))+(h^{\prime },k^{\prime
})(V_{w}(\lnot \alpha )-(z,w)) \\
&=&\left\{ 
\begin{array}{l}
(h,k)((1,0)-(x,y))+(h^{\prime },k^{\prime })((0,1)-(z,w)) \\ 
(h,k)((0,0)-(x,y))+(h^{\prime },k^{\prime })((0,0)-(z,w)) \\ 
(h,k)((0,1)-(x,y))+(h^{\prime },k^{\prime })((1,0)-(z,w))
\end{array}
\right.  \\
&=&\left\{ 
\begin{array}{l}
(h(1-x),-ky)+(-h^{\prime }z,k^{\prime }(1-w)) \\ 
(-hx,-ky)+(-h^{\prime }z,-k^{\prime }w) \\ 
(-hx,k(1-y))+(h^{\prime }(1-z),-k^{\prime }w)
\end{array}
\right.  \\
&=&\left\{ 
\begin{array}{l}
(h-(hx+h^{\prime }z),k^{\prime }-(ky+k^{\prime }w)) \\ 
(hx+h^{\prime }z,ky+k^{\prime }w) \\ 
(h^{\prime }-(hx+h^{\prime }z),k-(ky+k^{\prime }w)).
\end{array}
\right. 
\end{eqnarray*}
Then, by the above lemma, we can choose $h,h,h^{\prime },k^{\prime }$ such
that: i) $hx+h^{\prime }z=ky+k^{\prime }w$, and ii) $h<k^{\prime }$ and $%
h^{\prime }>k$. So in the first and third case we have a point in $\delta
^{-}$ and in the second case we have a point in $\delta $. This proves that $%
\ [B](w)\in \delta ^{-}\cup \delta $ and $B$ is a Weak Dutch Book.

Axiom 4. We must prove that $n\models \alpha $ implies $(0,0)\preccurlyeq
b(\alpha )$. We suppose that $b(\alpha )=(x,y)$ and $(0,0)\npreceq (x,y)$.
We set $B=\{(a,(x.y),(0,-1))\}$. From the hypothesis $n\models \alpha $ we
know that, for all $w\in K^{n}$, $V_{w}(n)\leq V_{w}(\alpha )$. As $%
V_{w}(n)=(0,0)$, we have $V_{w}(\alpha )=(0,0)$ or $V_{w}(\alpha )=(1,0)$.
Then, for all $w\in K^{n}$, 
\begin{equation*}
\lbrack B](w)=(0,1)(V_{w}(\alpha )-(x,y))=\left\{ 
\begin{array}{ll}
(0,-1)((1,0)-(x,y)) & \text{if }V_{w}(\alpha )=(1,0)\text{,} \\ 
(0,-1)((0,0)-(x,y)) & \text{if }V_{w}(\alpha )=(0,0)\text{.}
\end{array}
\right. 
\end{equation*}
In both cases we have $[B](w)=(0,y)$. As $(x,y)$ is a partial probability
value, $0\leq y\leq 1$; if $y=0$ then $(0,0)\preccurlyeq (x,y)$, but by
hypothesis $(0,0)\npreceq (x,y)$, so $0<y$ . Then we have $(0,y)\prec (0,0)$%
, so $[B](w)\in \delta ^{-}$ and $B$ is a Dutch Book.
\end{proof}

\section{Equivalence}

We conclude our work with some remarks about the concept of bet. We start
with classical bets. If $b:F\rightarrow \lbrack 0,1]$ is to represent the
degree of belief of a decision-maker, logically equivalent formulas of
classical logic should receive the same probability value. We can show that
this is the case, when there is no Dutch Book for $b$. We suppose that $%
\alpha \equiv \beta $ in classical logic and $b(\alpha )\neq b(\beta )$. If $%
b(\beta )<b(\alpha )$ we set $B=\{(\alpha ,b(\alpha ),1),(\beta ,b(\beta
),-1)\}$, so we have $[B](w)=V_{w}(\alpha )-V_{w}(\beta )+b(\beta )-b(\alpha
)=b(\beta )-b(\alpha )<0$ and $B$ is a Dutch Book for $b$. If $b(\alpha
)<b(\beta )$ we choose a stake $-1$ for the bet on $\alpha $ and a stake $1$
for the bet on $\beta $.

From this observation we can see that the concept of bet introduced above is
substantially equivalent to another that is common in the literature (see,
for instance, \cite{halpern2003} p. 20-23). In fact, we can define a bet as
a triple $(U,x,r)$, where $U\subseteq 2^{n}$, $x\in \lbrack 0,1]$ and $r\in
R $, so in this way we are betting on an event, a set of possible world $U$,
instead of betting on a formula $\alpha $. If we denote with $\mathcal{B}$
the set of all bets of kind $(\alpha ,x,r)$ and define a relation $(\alpha
,x,r)\sim (\beta ,x,r)$ iff $\alpha \equiv \beta $, then $\sim $ is an
equivalence relation and we can take the quotient $\mathcal{B}/\sim $. Now
there is a bijection $f$ between $\mathcal{B}/\sim $ and the set of all bets
of kind $(U,x,r)$: just set $f(|(\alpha ,x,r)|)=(M(\alpha ),x,r)$, where $%
M(\alpha )$ denotes the set of possible worlds that is the meaning of $%
\alpha $, as defined in \ref{parclassbet}.\ The function $f$ is well defined
because if $(\alpha ,x,r)\sim (\beta ,x,r)$ then $\alpha \equiv \beta $ and $%
M(\alpha )=M(\beta )$. The function is obviously injective and is surjective
because every $U\subseteq 2^{n}$ is $M(\alpha )$ for some $\alpha $. (This
follows from the theorem on disjunctive normal form of classical logic.) As
a consequence of the observation above, all bets in the class $|(\alpha
,x,r)|$ behave in the same way with respect to $b$, when no Dutch Book is
possible for $b$, so we can attach the probability value $b(\alpha )$ to the
class $|(\alpha ,x,r)|$ and this value can be transferred to $M(\alpha )$.

The same problem can be posed for partial bets: if $b:F^{\ast }\rightarrow T$
is to represent the \ partial degree of belief of a decision-maker,
logically equivalent formulas of Kleene logic should receive the same
partial probability value. The following theorem proves that this is the
case when there is no Dutch Book for $b$.

\begin{theorem}
Let $b:F^{\ast }\rightarrow T$ and suppose that there is no Dutch Book for $b
$. If $\alpha \equiv \beta $ in Kleene logic then $b(\alpha )=b(\beta )$.
\end{theorem}

\begin{proof}
We assume that $\alpha \equiv \beta $ and $b(\alpha )=(x,y)\neq
(z,w)=b(\beta )$, then we show that there is a Dutch Book for $b$.

Case 1, $z-x<w-y$. We set $B=\{(\alpha ,(x,y),(1,1)),(\beta ,(z,w),(-1,-1))\}
$ Then 
\begin{eqnarray*}
\lbrack B](w) &=&(1,1)(V_{w}(\alpha )-(x,y))+(-1,-1)(V_{w}(\beta )-(z,w)) \\
&=&(1,1)V_{w}(\alpha )-(1,1)(x,y))+(-1,-1)V_{w}(\beta )-(-1,-1)(z,w) \\
&=&-(x,y)-(-z,-w) \\
&=&(z-x,w-y).
\end{eqnarray*}
As $(z-x,w-y)\prec (0,0)$, we have $[B](w)\in \delta ^{-}$ and $B$ is a
Dutch Book.

Case 2, $w-y<z-x$. We set $B=\{(\alpha ,(x,y),(-1,-1)),(\beta ,(z,w),(1,1))\}
$ so that $[B](w)=(x-z,y-w)\in \delta ^{-}$.

Case 3, $z-x=w-y$. Firstly we observe that $z-x\neq 0$. (If $z-x=0$ then $%
w-y=0$ and $x$, $y$, $z$ and $w$ are all $0$, that is contrary to our
hypothesis $(x,y)\neq (z,w)$.) Then we distinguish two cases.

Subcase a), $z-x<0$. Then $w-y<0$ and $y-w>0$, so $z-x<y-w$ and $($ $%
z-x<y-w)\in \delta ^{-}$. We obtain a Dutch Book choosing $B=\{(\alpha
,(x,y),(1,-1)),(\beta ,(z,w),(-1,1))\}$, because in this case $%
[B](w)=(z-x,y-w)$.

Subcase b), $0<z-x$. We set $B=\{(\alpha ,(x,y),(-1,1)),(\beta
,(z,w),(1,-1))\}$.
\end{proof}

Now we can take \ equivalence classes of partial bets, as we did in the case
of classical bets. As all bets in the class $|(\alpha ,x,r)|$ behave in the
same way with respect to $b$, when no Dutch Book is possible for $b$, we can
attach the partial probability value $b(\alpha )$ to the class $|(\alpha
,(x,x^{\prime }),(r,r^{\prime }))|$, but here the analogy breaks off. \ If
we define a concept of bet where the formula $\alpha $ is replaced by a
partial set of possible worlds, i.e. we define a bet as a triple $%
((U,U^{\prime }),(x,x^{\prime }),(r,r^{\prime }))$, where $(U,U^{\prime
})\subseteq \mathcal{D}(K^{n})$, $(x,x^{\prime })\in T$ and $(r,r^{\prime
})\in R^{2}$, we obtain a more comprehensive concept of partial bet: there
are bets on partial events that cannot be simulated by a bet on a sentence.
In fact, we may define a function $f(|(\alpha ,(x,x^{\prime }),(r,r^{\prime
})|)=(M(\alpha ),(x,x^{\prime }),(r,r^{\prime }))$ as above, where $M(\alpha
)$ is the partial set of positive and negative models of $\alpha $ defined
in \ref{parparsub}, but now $f$ is strictly iniective, because there are
partial sets in $\mathcal{D}(K^{n})$ that cannot be denoted by a formula $%
\alpha $. We end this work with a proof of this theorem.

The elements of $K$, that have been up to now understood as truth values,
may also be conceived as values assigned to the amount of information
carried by a proposition, to the degree of exactness of an assertion. (In
this section we take the original definition of $K$ as $\{0,n,1\}$.) To this
end, we define on $K$ a partial order $\trianglelefteq $ setting $%
x\trianglelefteq y$ iff $x=y$ or ($x=n$ and $y=0$) or ($x=n$ and $y=1$).
This partial order can be pointwise extended to $K^{n}$, so given $s$, $t\in
K^{n}$, $s\trianglelefteq t$ iff for all $i<n$, $s_{i}\trianglelefteq t_{i}$%
. The partially ordered set $(K^{n},\trianglelefteq )$ has a minimum $%
\overline{n}$, the sequence that takes always $n$ as value, and $2^{n}$
maximal elements, the sequences that take always $0$ or $1$ as value.

If $s$ and $t$ are seen as possible worlds, then $s\trianglelefteq t$ means
that the situation represented by $t$ is at least so defined as the one
described by $s$. What in $s$ has been definitely settled (marked with $0$
or $1$) remains unchanged in the passage to $t$; what has not been
definitely settled in $s$ (marked with $n$) may be (positively or
negatively) settled in $t$. Whereas $V_{s}(\alpha )\leq V_{s}(\alpha )$
means that there is no loss in truth-value in the passage from $\alpha $ to $%
\beta $, $V_{s}(\alpha )\trianglelefteq V_{s}(\alpha )$ means that there is
no loss of information.

\begin{theorem}
For all $s$, $t\in K^{n}$ and all $\alpha \in F^{\ast }$,

\begin{enumerate}
\item  if $s\trianglelefteq t$ then $V_{s}(\alpha )\trianglelefteq
V_{t}(\alpha )$,

\item  if $s\trianglelefteq t$ and $s\in M(\alpha )_{i}$ then $t\in M(\alpha
)_{i}$, for $i=0,1$.
\end{enumerate}
\end{theorem}

\begin{proof}
1. Firstly, we observe that the functions $\lnot $, $\wedge $, $\vee $
defined on $K$ (see \ref{parparsub}) are isotone. (This can be easily
verified by the reader.) Then we suppose $s\trianglelefteq t$ and prove $%
V_{s}(\alpha )\trianglelefteq V_{t}(\alpha )$ by induction on $\alpha $. If $%
\alpha =p_{i}$ then $V_{s}(p_{i})\trianglelefteq V_{t}(p_{i})$ holds,
because $s_{i}\trianglelefteq t_{i}$. If $\alpha =\beta \wedge \gamma $ then
by induction hypothesis we have $V_{s}(\beta )\trianglelefteq V_{t}(\beta )$
and $V_{s}(\gamma )\trianglelefteq V_{t}(\gamma )$, so the result follows by
isotonicity of $\wedge $ with respect to $\trianglelefteq $. The same proof
can be given for $\vee $ and $\lnot $.

2. \ As we have seen in \ref{parparsub}, $s\in M(\alpha )_{0}$ iff $%
Val_{s}(\alpha )=1$. So from our hypothesis $s\trianglelefteq t$ and point
1) above, $s\in M(\alpha )_{0}$ implies $Val_{t}(\alpha )=1$ and $t\in
M(\alpha )_{0}$. The same proof holds for $M(\alpha )_{1}$.
\end{proof}

\begin{corollary}
The function $M:F^{\ast }\rightarrow D(K^{n})$ is not surjective.
\end{corollary}

\begin{proof}
For all $s\in K^{n}$ we have $s\trianglelefteq \overline{n}$ , so if $%
\overline{n}\in M(\alpha )_{i}$ then, for all $s\in K^{n}$, $s\in M(\alpha
)_{i}$, by the above theorem, and $M(\alpha )_{i}=K^{n}$. Let $(X,Y)$ be a
Boolean partial set, different from $(\emptyset ,K^{n})$ and $%
(K^{n},\emptyset )$. Then we have $\overline{n}\in $ $X$ or $\overline{n}\in
Y$ and $(X,Y)\neq M(\alpha )$, for all $\alpha $.
\end{proof}

\end{document}